\documentclass[a4paper, reqno, 12pt]{amsart}

\usepackage[usenames,dvipsnames]{color}
\usepackage{amsthm,amsfonts,amssymb,amsmath,amsxtra}
\usepackage[all]{xy}
\SelectTips{cm}{}
\usepackage{xr-hyper}
\usepackage[colorlinks=
   citecolor=Black,
   linkcolor=Red,
   urlcolor=Blue]{hyperref}
\usepackage{verbatim}

\usepackage[margin=1.25in]{geometry}
\usepackage{mathrsfs}

\RequirePackage{xspace}
\RequirePackage{etoolbox}
\RequirePackage{varwidth}
\RequirePackage{enumitem}
\RequirePackage{tensor}
\RequirePackage{mathtools}
\RequirePackage{longtable}
\RequirePackage{multirow}

\def\ge{\geqslant}
\def\le{\leqslant}
\def\a{\alpha}
\def\b{\beta}

\def\G{\Gamma}

\def\D{\Delta}

\def\s{\sigma}
\def\t{\tau}

\def\k{\kappa}
\def\l{\lambda}

\def\i{^{-1}}

\def\<{\langle}
\def\>{\rangle}

\newcommand{\bG}{\mathbf G}

\newcommand{{\BG}}{\ensuremath{\mathbb {G}}\xspace}

\newcommand{{\BK}}{\ensuremath{\mathbb {K}}\xspace}

\newcommand{\BN}{\ensuremath{\mathbb {N}}\xspace}

\newcommand{\BQ}{\ensuremath{\mathbb {Q}}\xspace}

\newcommand{\BS}{\ensuremath{\mathbb {S}}\xspace}

\newcommand{\CB}{\ensuremath{\mathcal {B}}\xspace}

\newcommand{\CI}{\ensuremath{\mathcal {I}}\xspace}

\newcommand{\CO}{\ensuremath{\mathcal {O}}\xspace}

\newcommand{\Ad}{{\text{Ad}}}

\DeclareMathOperator{\Adm}{Adm}

\DeclareMathOperator{\Gal}{Gal}

\newcommand{\id}{\ensuremath{\text{id}}\xspace}

\DeclareMathOperator{\rank}{rank}

\def\tW{\tilde W}

\newtheorem{theorem}{Theorem}
\newtheorem{proposition}[theorem]{Proposition}
\newtheorem{lemma}[theorem]{Lemma}

\newtheorem{corollary}[theorem]{Corollary}

\theoremstyle{definition}

\newtheorem{remark}[theorem]{Remark}

\numberwithin{equation}{section}
\numberwithin{theorem}{section}

\setitemize[0]{leftmargin=*,itemsep=\the\smallskipamount}
\setenumerate[0]{leftmargin=*,itemsep=\the\smallskipamount}

\renewcommand{\to}{%
   \ifbool{@display}{\longrightarrow}{\rightarrow}%
   }
\let\shortmapsto\mapsto
\renewcommand{\mapsto}{%
   \ifbool{@display}{\longmapsto}{\shortmapsto}%
   }
\newlength{\olen}
\newlength{\ulen}
\newlength{\xlen}
\newcommand{\xra}[2][]{%
   \ifbool{@display}%
      {\settowidth{\olen}{$\overset{#2}{\longrightarrow}$}%
       \settowidth{\ulen}{$\underset{#1}{\longrightarrow}$}%
       \settowidth{\xlen}{$\xrightarrow[#1]{#2}$}%
       \ifdimgreater{\olen}{\xlen}%
          {\underset{#1}{\overset{#2}{\longrightarrow}}}%
          {\ifdimgreater{\ulen}{\xlen}%
             {\underset{#1}{\overset{#2}{\longrightarrow}}}
             {\xrightarrow[#1]{#2}}}}%
      {\xrightarrow[#1]{#2}}
   }
\makeatother
\newcommand{\xyra}[2][]{%
   \settowidth{\xlen}{$\xrightarrow[#1]{#2}$}%
   \ifbool{@display}%
      {\settowidth{\olen}{$\overset{#2}{\longrightarrow}$}%
       \settowidth{\ulen}{$\underset{#1}{\longrightarrow}$}%
       \ifdimgreater{\olen}{\xlen}%
          {\mathrel{\xymatrix@M=.12ex@C=3.2ex{\ar[r]^-{#2}_-{#1} &}}}%
          {\ifdimgreater{\ulen}{\xlen}%
             {\mathrel{\xymatrix@M=.12ex@C=3.2ex{\ar[r]^-{#2}_-{#1} &}}}
             {\mathrel{\xymatrix@M=.12ex@C=\the\xlen{\ar[r]^-{#2}_-{#1} &}}}}}%
      {\mathrel{\xymatrix@M=.12ex@C=\the\xlen{\ar[r]^-{#2}_-{#1} &}}}%
   }
\makeatletter
\newcommand{\xla}[2][]{%
   \ifbool{@display}%
      {\settowidth{\olen}{$\overset{#2}{\longleftarrow}$}%
       \settowidth{\ulen}{$\underset{#1}{\longleftarrow}$}%
       \settowidth{\xlen}{$\xleftarrow[#1]{#2}$}%
       \ifdimgreater{\olen}{\xlen}%
          {\underset{#1}{\overset{#2}{\longleftarrow}}}%
          {\ifdimgreater{\ulen}{\xlen}%
             {\underset{#1}{\overset{#2}{\longleftarrow}}}
             {\xleftarrow[#1]{#2}}}}%
      {\xleftarrow[#1]{#2}}
   }
\newcommand{\isoarrow}{%
   \ifbool{@display}{\overset{\sim}{\longrightarrow}}{\xrightarrow\sim}%
   }
   
\begin{document}

\title[Dimension of $X(\mu, b)$]{Dimension formula of the affine Deligne-Lusztig variety $X(\mu, b)$}

\author{Xuhua He}
\address{X.~H., The Institute of Mathematical Sciences and Department of Mathematics, The Chinese University of Hong Kong, Shatin, N.T., Hong Kong.}
\email{xuhuahe@math.cuhk.edu.hk}
\author{Qingchao Yu}
\address{Q.~Y., Academy of Mathematics and Systems Science, Chinese Academy of Sciences, Haidian, Beijing.}
\email{yuqingchao15@mails.ucas.ac.cn}

\thanks{}

\keywords{Affine Deligne-Lusztig varieties}
\subjclass[2010]{20G25, 11G25, 20F55}


\begin{abstract}
The study of certain union $X(\mu, b)$ of affine Deligne-Lusztig varieties in the affine flag varieties arose from the study of Shimura varieties with Iwahori level structure. In this paper, we give an explicit formula of $\dim X(\mu, b)$ for sufficiently large dominant coweight $\mu$. 
\end{abstract}

\maketitle

\section{Introduction}

\subsection{Motivation} The notion of affine Deligne-Lusztig variety was first introduced by Rapoport in \cite{Ra} and plays an important role in arithmetic geometry and the Langlands program. In this paper, the main characters are the affine Deligne-Lusztig varieties $X_w(b)$ in the affine flag varieties and certain union $X(\mu, b)$ of these varieties. 

Let us first explain the notations. Let $F$ be a nonarchimedean local field and $\breve F$ be the completion of its maximal unramified extension. Let $\bG$ be a connected reductive group over $F$. Let $\breve \CI$ be the standard Iwahori subgroup of $\bG(\breve F)$. Let $w$ be an element in the Iwahori-Weyl group $\tW$ and $\mu$ be a conjugacy class of cocharacters of $\bG$ (over the algebraic closure $\bar F$). Let $\Adm(\mu)$ be the $\mu$-admissible subset of $\tW$. See \S\ref{sec:adlv} for more details.  
Then for $b \in \bG(\breve F)$, we define
\begin{gather*}
X_w(b)=\{g \breve \CI \in \bG(\breve F)/\breve \CI; g \i b \s(g) \in \breve \CI \dot w \breve \CI\}; \\
X(\mu, b)=\bigsqcup_{w \in \Adm(\mu)} X_w(b)=\{g \breve \CI \in \bG(\breve F)/\breve \CI; g \i b \s(g) \in \breve \CI \Adm(\mu) \breve \CI\}. 
\end{gather*}

They are subschemes, locally of finite type, of the affine flag variety
(in the usual sense in equal characteristic; in the sense of Zhu [62] in mixed characteristic). In the case where the pair $(\bG, \mu)$ is a Shimura datum, $X(\mu, b)$ is the group-theoretic model for the Newton stratum $S_{[b]}$ corresponding to the $\s$-conjugacy class $[b]$ of $\bG(\breve F)$, inside the special fiber of the Shimura variety $Sh$ with Iwahori level structure. In particular, the dimension of $X(\mu, b)$ in this case is expected to be equal to the dimension of the Newton stratum $S_{[b]}$ minus the dimension of the central leaves. See \cite{HR} and \cite[\S 7]{GHN}. 

\subsection{Previous results} The first fundamental question in the study of $X(\mu, b)$ is the nonemptiness pattern, i.e., for which $b$, $X(\mu, b)$ is nonempty. Kottwitz and Rapoport in \cite{KR03} and \cite{Ra} conjectured that the nonemptiness pattern of $X(\mu, b)$ is given by the ``Mazur's inequality''. Some partial results were obtained by Rapoport and Richartz in \cite{RR} and by Wintenberger in \cite{Wi}. The conjecture was finally proved in \cite{He-KR}. 

The next fundamental question is the dimension formula of $X(\mu, b)$. In \cite{GY}, G\"ortz and Yu studied the dimension of $X(\mu, b)$ for basic $b$ in the Siegel case and obtained some partial results. In \cite{GHN}, G\"ortz, Nie and the first author introduced the fully Hodge-Newton decomposable cases, in which the dimension and the geometric structures of $X(\mu, b)$ are easy to describe. Other than those cases, very little was known about the dimension of $X(\mu, b)$. 

Let us make a digression and give a quick review on the dimension formula of an individual affine Deligne-Lusztig variety $X_w(b)$. 

In \cite[Conjecture 9.5.1]{GHKR2}, G\"ortz, Haines, Kottwitz and Reuman made two influential conjectures on the dimension formula of $X_w(b)$.  

\begin{itemize}
\item For the basic $\s$-conjugacy class $[b]$, they gave a conjecture in \cite[Conjecture 9.5.1 (a)]{GHKR2} on the dimension formula for $X_w(b)$ for $w$ in the so-called Shrunken Weyl chamber. This conjecture was established in \cite{He14}. 

\item For arbitrary $\s$-conjugacy class $[b]$, they gave a conjecture in \cite[Conjecture 9.5.1 (b)]{GHKR2} on the difference between the dimension of $X_w(b)$ and $X_w(b_{basic})$ for sufficiently large $w$. Here $[b_{basic}]$ is the associated basic $\s$-conjugacy class. This conjecture was established very recently in \cite{He20} under the assumption that $w$ is in the Shrunken Weyl chamber. 
\end{itemize}

\subsection{Main result} 
Rapoport in \cite{Ra} predicted ``whereas the individual affine Deligne-Lusztig varieties are very difficult to understand, the situation seems to change radically when we form a suitable finite union (i.e. $X(\mu, b)$) of them.''

As we have now a pretty good understanding on the dimension of an individual affine Deligne-Lusztig variety $X_w(b)$ for sufficiently large $w$, it is natural to expect a simple dimension formula of $X(\mu, b)$ for sufficiently large $\mu$. In this paper, we establish such a formula for quasi-split groups. For simplicity, we focus on the quasi-simple and split groups in the introduction. The explicit formula is as follows

\begin{theorem}[Theorem \ref{main}]\label{thmA}
Suppose that $\bG$ is quasi-simple and split over $F$ and $\mu$ is ``sufficiently regular''. Let $[b] \in B(\bG, \mu)$ with $\mu \ge \nu_b+2\rho^\vee$. Then $$\dim X(\mu, b)=\<\mu-\nu_b, \rho\>-\frac{1}{2}\text{def}_{\bG}(b)+\frac{\ell(w_0)-\ell_R(w_0)}{2}.$$
\end{theorem}

We refer to \S\ref{sec:notation}, \S\ref{sec:vir} and \S\ref{sec:refl} for the notations in this formula. We also give an explicit bound on $\mu$ so that the above dimension formula holds. 


\subsection{Strategy}

Recall that $X(\mu, b)=\sqcup_{w \in \Adm(\mu)} X_w(b)$.

The first main ingredient is an explicit description of most elements in the admissible set $\Adm(\mu)$. Note that the admissible set is a quite complicated combinatorial object. Partially motivated by the work of Haines and Ng\^{o} \cite{HN}, in \cite{HH} Haines and the first author gave an explicit description of all the elements $w$ in $\Adm(\mu)$. However, such description is not easy to use in order to calculate the dimension of $X_w(b)$. In \S\ref{sec:qbg-adm}, we give a different description for the ``sufficiently large'' elements inside $\Adm(\mu)$, via the quantum Bruhat graph introduced by Fomin, Gelfand and Postnikov in \cite{FGP}. 

The second main ingredient is the dimension formula of an individual affine Deligne-Lusztig variety $X_w(b)$ for sufficiently large $w$, which is established in \cite{He14} for basic $b$ and \cite{He20} for arbitrary $b$. Note that we do not have an explicit formula of $\dim X_w(b)$ for all $w$. However, we will show that in the end, those ``not-sufficiently-large'' $w$ do not make contribution to the top dimensional irreducible components of $X(\mu, b)$. 

Based on the two ingredients that we discussed above, we reduce the problem on $\dim X(\mu, b)$ to a problem on the quantum Bruhat graphs, which is a combinatorial problem on the finite Weyl groups. Further consideration reduces this problem to the problem of determining
$$\max\{\ell(x); x \in W_0, x \le x w_0\}.$$

This is a problem on the finite Weyl group, which is of independent interest. Jeff Adams and David Vogan helped us to compute this number for exceptional groups via the Atlas software. Based on these data, we finally realize that this number equals to $\frac{\ell(w_0)-\ell_R(w_0)}{2}$ and we then find a proof for it. 

\smallskip

{\bf Acknowledgment:} We thank Jeff Adams and David Vogan for their help on the computation for exceptional groups. We thank George Lusztig for pointing out an intriguing connection between the reflection length (and thus the dimension formula in Theorem \ref{thmA}) with the dimension of certain Springer fibers. We thank Ulrich G\"ortz and Elizabeth Mili\'cevi\'c for many useful comments on a preliminary version of the paper. 

XH is partially supported by a start-up grant and by funds connected with Choh-Ming Chair at CUHK. QY is supported by the National Natural Science Foundation of China (Grant No. 11621061).

\section{Preliminary}

\subsection{Notations}\label{sec:notation} Recall that $F$ is a non-archimedean local field and $\breve F$ is the completion of the maximal unramified extension of $F$. We write $\Gamma$ for $\Gal(\overline F/F)$, and write $\Gamma_0$ for the inertia subgroup of $\Gamma$.

Let $\bG$ be a quasi-split connected reductive group over $F$. Let $\s$ be the Frobenius morphism of $\breve F/F$. We write $\breve G$ for $\bG(\breve F)$. We use the same symbol $\s$ for the induced Frobenius morphism on $\breve G$. Let $S$ be a maximal $\breve F$-split torus of $\bG$ defined over $F$, which contains a maximal $F$-split torus. Let $T$ be the centralizer of $S$ in $\bG$. Then $T$ is a maximal torus. 

Let $\mathcal A$ be the apartment of $\bG_{\breve F}$ corresponding to $S_{\breve F}$. We fix a $\s$-stable alcove $\mathfrak a$ in $\mathcal A$, and let $\breve \CI \subset \breve G$ be the Iwahori subgroup corresponding to $\mathfrak a$. Then $\breve \CI$ is $\s$-stable. 

We denote by $N$ the normalizer of $T$ in $\bG$. The \emph{Iwahori--Weyl group} (associated to $S$) is defined as $$\tW= N(\breve F)/T(\breve F) \cap \breve \CI.$$ For any $w \in \tW$, we choose a representative $\dot w$ in $N(\breve F)$. The action $\s$ on $\breve G$ induces a natural action of $\s$ on $\tW$, which we still denote by $\s$. 

We fix a $\s$-stable special vertex of the base alcove $\mathfrak a$. Let $W_0=N(\breve F)/T(\breve F)$ be the relative Weyl group. Then we have the splitting $$\tW=X_*(T)_{\G_0} \rtimes W_0=\{t^{\underline \l} w; \underline \l \in X_*(T)_{\G_0}, w \in W_0\}.$$ 

Since $\bG$ is quasi-split over $F$, $\s$ acts naturally on $X_*(T)_{\G_0}$ and on $W_0$. Let $\Phi^+$ be the set of positive relative roots and $\D$ be the set of relative simple roots determined by the dominant Weyl chamber. Then $\s(\D)=\D$. Let $w_0$ be the longest element in $W_0$. Let $\rho$ be the dominant weight with $\<\a^\vee, \rho\>=1$ for any $\a \in \D$ and $\rho^\vee$ be the dominant coweight with $\<\rho^\vee, \a\>=1$ for any $\a \in \D$. 

We denote by $\ell$ the length function on $\tW$ and on $W_0$, and by $\le$ the Bruhat order on $\tW$ and on $W_0$. Let $\tilde \BS$ be the set of simple reflections in $\tW$ and $\BS \subset \tilde \BS$ be the set of simple reflections in $W_0$. 

\subsection{The $\s$-conjugacy classes of $\breve G$}

We define the $\s$-conjugation action on $\breve G$ by $g \cdot_\s g'=g g' \s(g) \i$. Let $B(\bG)$ be the set of $\s$-conjugacy classes on $\breve G$. The classification of the $\s$-conjugacy classes is obtained by Kottwitz in \cite{Ko1} and \cite{Ko2}. Any $\s$-conjugacy class $[b]$ is determined by two invariants: 
\begin{itemize}
	\item The element $\k([b]) \in \pi_1(\bG)_{\s}$; 
	
	\item The Newton point $\nu_b \in \big((X_*(T)_{\Gamma_0, \BQ})^+\big)^{\langle\sigma\rangle}$. 
\end{itemize}

Here $-_\s$ denotes the $\s$-coinvariants, 
$(X_*(T)_{\Gamma_0, \BQ})^+$ denotes the intersection of  $X_*(T)_{\Gamma_0}\otimes \BQ=X_*(T)^{\Gamma_0}\otimes \BQ$ with the set $X_*(T)_\BQ^+$ of dominant elements in $X_*(T)_\BQ$. 

Let $\mu$ be a conjugacy class of cocharacters of $\bG$. We choose a $\breve F$-rational dominant representative $\mu_+$ in this conjugacy class and denote by $\underline \mu$ the image in $X_*(T)_{\Gamma_0} $ of  $\mu_+$. See \cite[\S 2.2]{GHR} for more details.  

The set of \emph{neutrally acceptable} $\s$-conjugacy classes is defined by 
$$
B(\bG, \mu)=\{[b] \in B(\bG); \k([b])=\k(\mu), \nu_b \le \mu^\diamond\},
$$
where $\mu^\diamond \in X_*(T)^{\G_0} \otimes \BQ\cong X_*(T)_{\G_0} \otimes \BQ$ is the average of the $\s$-orbit on $\mu_+$.

\subsection{The union of affine Deligne-Lusztig varieties}\label{sec:adlv}
Let $Fl=\breve G/\breve \CI$ be the {\it affine flag variety}. For any $b \in \breve G$ and $w \in \tW$, we define the corresponding {\it affine Deligne-Lusztig variety} in the affine flag variety $$X_w(b)=\{g \breve \CI \in \breve G/\breve \CI; g \i b \s(g) \in \breve \CI \dot w \breve \CI\} \subset Fl.$$

The {\it $\mu$-admissible set} is defined by
$$
\Adm(\mu)=\{w \in \tilde W; w \le t^{x(\underline \mu)} \text{ for some }x \in W_0\} ,
$$
cf.~\cite{Ra}. 

For any $b \in \breve G$, we set $$X(\mu, b)=\bigsqcup_{w \in \Adm(\mu)} X_w(b)=\{g \breve \CI \in \breve G/\breve \CI; g \i b \s(g) \in \breve \CI \Adm(\mu) \breve \CI\}.$$

The following result was conjectured by Kottwitz and Rapoport in \cite{KR03, Ra} and proved in \cite[Theorem A]{He-KR}.

\begin{theorem}\label{KR-conj}
    Let $b \in \breve G$.  Then $X(\mu, b) \neq \emptyset$ if and only if $[b] \in B(\bG, \mu)$.
\end{theorem}

\section{Quantum Bruhat graphs and admissible sets}\label{sec:qbg-adm}

\subsection{Quantum Bruhat graph} We first recall the quantum Bruhat graph introduced by Fomin, Gelfand and Postnikov in \cite{FGP}. By definition, a quantum Bruhat graph $\G_{\Phi}$ is a directed graph with 
\begin{itemize}
\item vertices given by the elements of $W_0$; 

\item upward edges $w\rightharpoonup w s_\a$ for some $\a \in \Phi^+$ with $\ell(w s_\a)=\ell(w)+1$;

\item downward edges $w \rightharpoondown w s_\a$ for some $\a \in \Phi^+$ with $\ell(w s_\a)=\ell(w)-\<\a^\vee, 2 \rho\>+1$. 
\end{itemize}

The weight of an upward edge is defined to be $0$ and the weight of a downward edge $w \rightharpoondown w s_{\a}$ is defined to be $\a^{\vee}$. The weight of a path in $\G_{\Phi}$ is defined to be the sum of weights of the edges in the path. For any $x,y\in W_0$, we denote by $d_\G(x, y)$ the minimal length among all paths in $\G_{\Phi}$ from $x$ to $y$. The following result is proved by Postnikov in \cite[Lemma 1]{Pos05}. 

\begin{lemma}\label{wt-x-y}
Let $x,y\in W_0$ be any elements. Then 
\begin{enumerate}
\item There exists a directed path in $\G_{\Phi}$ from $x$ to $y$.

\item All the shortest paths in $\G_{\Phi}$ from $x$ to $y$ have the same weight, which we denote by $\text{wt}(x, y)$. 

\item Any path in $\G_{\Phi}$ from $x$ to $y$ has weight $\ge \text{wt}(x,y)$.
\end{enumerate}
\end{lemma}

\subsection{Bruhat order on $\tW$}


In the rest of this section, we assume that the affine Dynkin diagram of $\tW$ is connected. For any dominant $\underline \l \in X_*(T)_{\G_0}$,  we define the depth\footnote{The meaning of depth is different from the one used in \cite{BBNW}.} of $\underline\l$ to be $$\text{depth}(\underline \l)=\text{min}\{\<\underline\l,\a\>; \a\in \Delta\}.$$

Lam and Shimozono in \cite[Proposition 4.4]{LS} proved the following relation between the quantum Bruhat graph and the Bruhat order for the elements in $\tW$ with  ``sufficiently regular'' translation part. The explicit bound of $\underline \l$ was given by Mili\'cevi\'c in \cite[Proposition 4.2]{Mil}. 

\begin{proposition}\label{qbg-bo} Assume that $W_0$ is an irreducible Weyl group. 
Suppose that $$ \text{depth}(\underline \l)\ge \begin{cases}
2\ell(w_0)+2, & \text{if } W_0 \text{ is not of type }G_2;\\
3\ell(w_0)+3, & \text{if } W_0 \text{ is of type }G_2.
\end{cases}
$$ 
Let $x, y \in W_0$ and $w=x t^{\underline \l} y$. Let $w' \in \tW$. Then $w' \lessdot w$ (i.e., $w' \le w$ and $\ell(w')=\ell(w)-1$) if and only if $w'$ is one of the following 
\begin{enumerate}
\item $w'=x s_\a t^{\underline \l} y$ for some $\a \in \Phi^+$ with $x s_\a \rightharpoonup x$; 

\item $w'=x s_\a t^{\underline \l-\a^\vee} y$ for some $\a \in \Phi^+$ with $x s_\a \rightharpoondown x$; 

\item $w'=x t^{\underline \l} s_\a y$ for some $\a \in \Phi^+$ with $y\i \rightharpoonup  y\i s_\a$; 

\item $w'=x t^{\underline \l-\a^\vee} s_\a y$ for some $\a \in \Phi^+$ with $y\i \rightharpoondown  y\i s_\a$. 
\end{enumerate}
\end{proposition}


We obtain the following description of certain elements in $\Adm(\mu)$. 

\begin{proposition}\label{adm-set}
Assume that $W_0$ is an irreducible Weyl group. 
Let $\underline \l \in X_*(T)_{\G_0}$ be dominant. Suppose that for any dominant $\underline \l'  \in X_*(T)_{\G_0}$ with $\underline \l \le \underline \l' \le \underline \mu$, we have \[\tag{*} \text{depth}(\underline\l')\ge \begin{cases}
2\ell(w_0)+2, & \text{if } W_0 \text{ is not of type }G_2;\\
3\ell(w_0)+3, & \text{if } W_0 \text{ is of type }G_2.
\end{cases}
\]
Then for any $x, y \in W_0$, $x t^{ \underline\l} y \in \Adm(\mu)$ if and only if there exists a path in the quantum Bruhat graph $\G_{\Phi}$ from $x$ to $y \i$ with weight $\underline \mu-\underline \l$. 
\end{proposition}
\begin{proof}
Let $w=xt^{\underline \l}y\in\tW $. Suppose $w\in\Adm(\mu)$, then there is some $u\in W_0$ such that $w\le ut^{\underline \mu}u\i$. Thus there exists a chain of elements $$w=w_1 \lessdot w_2 \lessdot \cdots \lessdot w_n=u t^{\underline \mu} u \i.$$ 

Let $\underline \l_i  \in X_*(T)_{\G_0}$ be dominant such that $w_i \in W_0 t^{\underline \l_i} W_0$. Since $w_i \lessdot w_{i+1}$, we have that $\underline \l_i \le \underline \l_{i+1}$ for $1 \le i \le n-1$. In particular, $\underline \l \le \underline \l_i \le \underline \mu$ for $1 \le i \le n-1$. By our assumption on $\underline\l$, Proposition \ref{qbg-bo} is applicable to all the covering relations $w_i \lessdot w_{i+1}$. Therefore, $\underline \mu-\underline \l$ equals to the sum of the weight of a path in $\G_{\Phi}$ from $x$ to $u$ and the weight of a path in $\G_{\Phi}$ from $u$ to $y \i$. The concatenation of these two paths gives a path from $x$ to $y \i$ with weight $\underline \mu-\underline \l$. 

Conversely, given a path in $\G_{\Phi}$ from $x$ to $y\i$ with weight $\underline \mu-\underline \l$, by Proposition \ref{qbg-bo}, one may construct a chain of elements $$w=w_1 \lessdot w_2 \lessdot \cdots \lessdot w_n=x t^{\underline \mu} x \i.$$ In particular, $w \in \Adm(\mu)$. 
\end{proof}

\subsection{An explicit bound on $\underline\mu$}\label{explicit-bound} Finally, we discuss the assumption (*) in Proposition \ref{adm-set}. 

Suppose that $$\text{depth}(\underline\mu)\ge \begin{cases}
4 \ell(w_0)+2, & \text{if }  W_0 \text{ is not of type }G_2;\\
5 \ell(w_0)+3, & \text{if }  W_0 \text{ is of type }G_2.\\
\end{cases}
$$ 
We claim that

(a) for any dominant cocharacter $\underline{\l} \le \underline\mu$ with $\<\underline\mu-\underline\l, \rho\> \le \ell(w_0)$, the assumption (*) in Proposition \ref{adm-set} is satisfied. 

Indeed, we have $\<\underline\mu-\underline\l', \rho\> \le \ell(w_0)$ for any $\underline \l'$ with $\underline\l \le \underline\l' \le \underline\mu$. Then $\underline\mu-\underline\l'=\sum_{\a \in \D} n_\a \a^\vee$ for some $n_\a \in \BN$ with $\sum_{\a \in \D} n_\a \le \ell(w_0)$. Then for any $\b \in \D$, $$\<\underline\l', \b\>=\<\underline\mu, \b\>-\sum_{\a \in \D} n_{\a} \<\a^\vee, \b\> \ge \<\underline\mu, \b\>-2 \ell(w_0).$$ 

Thus, $$\text{depth}(\underline\l') \ge  \begin{cases}
2\ell(w_0)+2, & \text{if }  W_0 \text{ is not of type }G_2;\\
3\ell(w_0)+3, & \text{if }  W_0 \text{ is of type }G_2.\\
\end{cases}$$
This proves the claim.

\section{A formula on the virtual dimension}

\subsection{Virtual dimension $d_w(b)$}\label{sec:vir}
We recall the definition of virtual dimension in \cite[\S10.1]{He14}. 

Note that any element $w \in \tW$ may be written in a unique way as $w=x t^  {\underline{\l}} y$ with $\underline\l$ dominant, $x, y \in W_0$ such that $t^{\underline{\l}} y \in {}^{\BS} \tW$. In this case, we set $$\eta_\s(w) = \s^{-1}(y) x.$$

Let $\mathbf J_b$ be the reductive group over $F$ with $$\mathbf J_b(F)=\{g \in \breve G; g b \s(g) \i=b\}.$$ The \emph{defect} of $b$ is defined by $$\text{def}_{\bG}(b)=\rank_F \bG-\rank_F \mathbf J_b.$$ Here for a reductive group $\mathbf H$ defined over $F$, $\rank_F$ is the $F$-rank of the group $\mathbf H$. 

The \emph{virtual dimension} is defined to be $$d_{w}(b)=\frac 12 \big( \ell(w) + \ell(\eta_\s(w)) -\text{def}_{\bG}(b)  \big)-\<\nu_b, \rho\>.$$

The following result is proved in \cite[Corollary 10.4]{He14} for residually split groups and proved in \cite[Theorem 2.30]{He-CDM} for the general case. 

\begin{theorem}\label{dim-vir}
Let $b \in \breve G$ and $w \in \tW$. Then $\dim X_{w}(b) \le d_{w}(b)$. 
\end{theorem}

Now we set $$d_{\Adm(\mu)}(b)=\max_{w \in \Adm(\mu)} d_{w}(b).$$ As a consequence of Theorem \ref{dim-vir}, we have 

\begin{corollary}\label{cor:4.2}
For any $b$, $\dim X(\mu, b) \le d_{\Adm(\mu)}(b)$. 
\end{corollary}

\subsection{Some formulas on $W_0$}\label{Subsection 4.2} In this subsection, we give some formulas needed in the proof of Proposition \ref{dim-formula}. 

By definition, if $w\rightharpoonup w s_\a$, then $\ell(w s_\a)=\ell(w)+1=\ell(w)-\<\text{wt}(w, w s_\a), 2 \rho\>+1$.  If $w\rightharpoondown w s_\a$, then $\ell(w s_\a)=\ell(w)-\<\a^\vee, 2 \rho\>+1=\ell(w)-\<\text{wt}(w, w s_\a), 2 \rho\>+1$. Therefore, for any $x, y \in W_0$, we have \[\tag{a} \ell(y)=\ell(x)-\<\text{wt}(x, y), 2 \rho\>+d_{\G}(x, y).\]

By definition, for any simple reflection $s$, we have an edge (either upward or downward) $x \to x s$. Thus for any $x, y \in W_0$, $d_{\G}(x, y) \le \ell(x \i y) \le \ell(w_0)$. Hence $\<\text{wt}(x, y), 2 \rho\>=\ell(x)-\ell(y)+d_{\G}(x, y) \le 2 \ell(w_0)$. Therefore 
\[\tag{b} \<\text{wt}(x, y), \rho\> \le \ell(w_0).\]

\begin{lemma}\label{reach-w}
The maximum of the set $$\{ \ell(\s\i(y)x) -   d_{\G}(x,y\i) ; x,y\in W_0 \}$$ can be achieved when $\s\i(y)x =w_0.$ In other words, 
$$ \max\{ \ell(\s\i(y)x) -   d_{\G}(x,y\i) ; x,y\in W_0 \}= \max\{\ell(w_0)-d_{\G}(x, \s(x)w_0); x\in W_0\}.$$
\end{lemma}

\begin{proof}
Suppose that $\s\i(y) x \ne w_0.$ Then there exists a simple reflection $s$ with $\s\i(y) xs>\s\i(y)x$. Since $s$ is a simple reflection, we have an edge (upward or downward) $xs  \to x.$ The concatenation of the edge $xs \to x$ and any path from $x$ to $y\i$ gives a path from $xs$ to $y\i.$ Hence $d_{\G}(xs,y\i)\le d_{\G}(x ,y\i)+1.$ It follows that $$\ell(\s\i(y)xs)-d_{\G}(xs,y\i)\ge  \ell(\s\i(y)x)-d_{\G}(x,y\i).$$ Thus, the maximum of $\ell(\s\i(y')x') - d_{\G}(x',(y') \i)$ can be achieved for some $x, y$ with $\s \i(y) x=w_0$. 
\end{proof}

\subsection{Formula of $d_{\Adm(\mu)}(b)$}
We have $W_0=W_{0,1} \times \cdots \times W_{0,l}$, where $W_{0, i}$ are the irreducible Weyl groups.  Any element $x \in W_0$ is of the form $x=(x_1, \ldots, x_l)$ with $x_i \in W_{0, i}$. We have $w_0=(w_{0,1}, \ldots, w_{0,l})$, where $w_{0, i}$ is the longest element of $W_{0, i}$. We write $\underline \mu$ as $\underline \mu=( \underline\mu_1, \ldots,\underline \mu_l)$. We say that $\underline \mu$ is {\it superregular} if for each $i$,
$$\text{depth}(\underline \mu_i)\ge  \begin{cases}
4\ell(w_{0,i})+2, & \text{if }  W_{0,i} \text{ is not of type }G_2;\\
5\ell(w_{0,i})+3, & \text{if }  W_{0,i} \text{ is of type }G_2.\end{cases}$$

It is worth pointing out that the superregularity condition here is weaker than the superregularity condition in \cite[Corollary 3.3]{Mil}, where it requires that $\text{depth}(\underline\mu_i)\ge 8\ell(w_{0,i})$ if $W_{0, i}$ is of classical type and $\text{depth}(\underline\mu_i)\ge 12\ell(w_{0,i})$ if $W_{0, i}$ is of exceptional type. See also \cite{MV} for some applications of that superregularity condition. 

\begin{proposition}\label{dim-formula}
Suppose that $\underline\mu$ is superregular. Then
$$d_{\Adm(\mu)}(b)= \<\underline \mu-\nu_b,\rho \> -\frac{1}{2}\text{def}_{\bG}(b)+\frac{1}{2}\ell(w_0) - \frac{1}{2}\min\{d_{\G}(x,\s(x)w_0);x\in W_0\}.$$
\end{proposition}
\begin{proof}
Let $w \in \Adm(\mu)$. Then $w$ is of the form $w=x t^  {\underline{\l}} y$ with $\underline\l$ dominant and $\underline\l\le \underline\mu$, $x, y \in W_0$ such that $t^{\underline{\l}} y \in {}^{\BS} \tW$. 

By \S\ref{Subsection 4.2}(a), we have \begin{align*}	 
	d_{w}(b) & = \frac{1}{2}\left(\ell(x)-\ell(y)+\ell(\s\i(y)x)+\<\underline \l,2\rho \>\right)-\frac{1}{2}\text{def}_{\bG}(b)-\<\nu_b,\rho \>\\ 
	  &=\frac{1}{2}(\ell(\s \i(y) x)-d_{\G}(x, y \i))+\<\text{wt}(x, y\i)+\underline \l, \rho\>-\frac{1}{2}\text{def}_{\bG}(b)-\<\nu_b,\rho \>.
\end{align*}

We show that 

(a) $\<\text{wt}(x, y)+\underline \l, \rho\> \le \<\underline \mu, \rho\>$. 

Write $\underline\l=(\underline\l_1,\ldots, \underline\l_l),$ $x=(x_1,\ldots,  x_l)$, $y=(y_1,\ldots,  y_l)$ and $\rho=(\rho_1, \ldots, \rho_l)$. Then $x_i t^{\underline\l_i}y_i\in\Adm(\mu_i)$. Write $\text{wt}(x, y\i)=(\underline \l'_1, \ldots, \underline \l'_l)$. For any $i$, if $\<\underline \mu_i -  \underline \l_i ,\rho\>  >\ell(w_{0,i})$, then by \S\ref{Subsection 4.2}(b), $\<\underline \l'_i+\underline \l_i, \rho_i\> \le \<\underline \mu_i, \rho_i\>$. If $\<\underline \mu_i -  \underline \l_i ,\rho\> \le \ell(w_{0,i})$, then by \S\ref{explicit-bound} and Proposition \ref{adm-set},  there exists a path in $\G_{\Phi}$ from $x_i$ to $y_i \i$ with weight $\underline \mu_i-\underline \l_i$. By Lemma \ref{wt-x-y} (3), $\text{wt}(x_i, y_i \i) \le \underline \mu_i-\underline \l_i$ and hence we also have $\<\underline \l'_i+\underline \l_i, \rho_i\> \le \<\underline \mu_i, \rho_i\>$. 

Thus (a) is proved.

By Lemma \ref{reach-w}, $\ell(\s \i(y) x)-d_{\G}(x, y \i) \le \ell(w_0)-\min\{d_{\G}(x,\s(x)w_0);x\in W_0\}$. Thus 
$$d_{\Adm(\mu)}(b) \le \<\underline \mu-\nu_b,\rho \> -\frac{1}{2}\text{def}_{\bG}(b)+\frac{1}{2}\ell(w_0) - \frac{1}{2}\min\{d_{\G}(x,\s(x)w_0);x\in W_0\}.$$

On the other hand, for any $x=(x_1, \ldots, x_l) \in W_0$,  we write $\s(x) w_0$ as $\s(x) w_0=(y_1, \ldots, y_l)$ and write $\underline \mu-\text{wt}(x, \s(x) w_0)$ as $(\underline \mu'_1, \ldots, \underline \mu'_l)$. By \S\ref{Subsection 4.2})(b), for any $i$, $\<\underline \mu_i-\underline \mu'_i, \rho_i\> \le \ell(w_{0, i})$. By \S\ref{explicit-bound} and Proposition \ref{adm-set}, $x_i t^{\underline \mu'_i} y_i \i \in \Adm(  \mu_i)$ and hence $x t^{\underline \mu-\text{wt}(x, \s(x) w_0)} w_0 \s(x) \i \in \Adm(\mu)$. By definition, 
$$d_{x t^{\underline \mu-\text{wt}(x, \s(x) w_0)} w_0 \s(x) \i}(b)=\<\underline\mu-\nu_b,\rho \>-\frac{1}{2}\text{def}_{\bG}(b)+\frac{1}{2}(\ell(w_0)-d_{\G}(x, \s(x)w_0)) .$$

This finishes the proof. 
\end{proof}

\section{Quantum Bruhat graphs and reflection lengths}

\subsection{Reflection length}\label{sec:refl} In this section, we assume that $W_0$ is a finite Coxeter group. This includes the finite Weyl groups we consider in the other sections in this paper, and also includes the finite Coxeter groups of type $H_3$, $H_4$ and $I_m$, which are of independent interest. 

Let $w \in W_0$. The {\it reflection length} of $w$ is the smallest number $l$ such that $w$ can be written as a product of $l$ reflections in $W$. We denote by $\ell_R(w)$ the reflection length of $w$. Since any simple reflection is a reflection, we have $\ell_R(w) \le \ell(w)$. For any subset $C$ of $W_0$, we write $$\ell_R(C)=\min\{\ell_R(w); w \in C\}.$$ 

Let $V$ be the reflection representation of $W_0$. By \cite[Lemma 2]{Car}, we have $\ell_R(w)=\dim V-\dim V^w$. In this paper, we are mainly interested in the reflection length of $w_0$. Below we list the explicit reflection length for all the simple groups. 

\begin{center}
\begin{tabular}{ |c|c|c|c|c|c|c|c|c|c|c|c|} 
 \hline
 Type & $A_n$ & $B_n/C_n$ & $D_n$ & $E_6$ & $E_7$ & $E_8$ & $F_4$ & $G_2$ &  $ H_3$  & $H_4$  & $I_m$\\ 
 \hline
 $\ell_R(w_0)$ & $\lceil \frac{n}{2} \rceil$ & $n$ & $2 \lfloor \frac{n}{2} \rfloor$ & $4$ & $7$ & $8$ & $4$ & $2$ & $3$  &  $4$  & $ \begin{cases}2,&\text{if }2\mid m \\     1,&\text{if } 2\nmid m\end{cases}$  \\
 \hline
\end{tabular}
\end{center}

\smallskip

Let $\CO$ be the $\s-$conjugacy class of $w_0$. In the rest of this subsection, we will discuss $\ell_R(\CO)$ for irreducible Weyl groups. 

If $\s=\Ad(w_0)$, then $\CO=\{w_0\}$ and thus $\ell_R(\CO)=\ell_R(w_0)$. 

If $\s=\id$, then $\CO$ is an ordinary conjugacy class and thus $\ell_R(z)=\ell_R(w_0)$ for any $z \in \CO$. 

If $\s$ is neither $\Ad(w_0)$ nor $\id$, then $(W_0, \s)$ is of type ${}^2 D_{2k}$, ${}^3 D_4$, ${}^2 F_4$ or ${}^2I_{2k}$ (here the superscript denotes the order of $\s$). In the ${}^2 D_{2k}$ case, we regard the elements in $W_0$ as the permutations on $\{\pm 1, \pm 2, \ldots, \pm 2m\}$ which commute with $-1$ and send even number of positive integers to negative integers. It is easy to see that $\ell_R(\CO)=2m-2$. In the ${}^2I_{2k}$ case, it is easy to see that $1\in\CO$ and $\ell_R(\CO)=0$. By direct computation, in the ${}^3 D_4$ case, $\ell_R(\CO)=2$ and in the ${}^2 F_4$ case, $1 \in \CO$ and $\ell_R(\CO)=0$.

Now we state the main result of this section. 

\begin{theorem}\label{explicit}
Let $W_0$ be a finite Coxeter group and $\s$ be a length-preserving group automorphism on $W_0$. Let $\CO$ be the $\s$-conjugacy class of $w_0$. Then 
$$\ell(w_0)-2\max\{\ell(x);x\le \s(x)w_0\}=\ell_R(\CO).$$
Moreover, if $W_0$ is a finite Weyl group, then we also have $$\ell_R(\CO)=\min\{d_{\G}(x, \s(x) w_0); x \in W_0\}.$$
\end{theorem}

\subsection{One direction}\label{sec:ineq} Let $x \in W_0$. Suppose that $x\le \s(x)w_0$. Let $m=\ell(\s(x)w_0)-\ell(x)=\ell(w_0)-2\ell(x)$. Then, there exists reflections $t_1,t_2,\ldots,t_m$ such that $\s(x)w_0=xt_1t_2\cdots t_m$. Hence $xw_0 \s(x)^{-1}= (xt_mx\i) \cdots(xt_1x\i)  $. By definition, we have $\ell_R(xw_0\s(x)^{-1})\le m$. Therefore
\begin{align*} \ell_R(\CO) & \le \min\{ \ell(w_0)-2\ell(x);x\le \s(x)w_0  \}\\ &=\ell(w_0)-2\max\{\ell(x);x\le \s(x)w_0\}.
\end{align*} 

In the following subsections, we shall prove that 
\[\tag{a} \ell(w_0)-2\max\{\ell(x);x\le \s(x)w_0\} \le \ell_R(\CO).\]

This will finish the proof of the first part of Theorem \ref{explicit}. 

Now suppose that $W_0$ is a finite Weyl group. Let $m=d_{\G}(x, \s(x) w_0)$. Then $\s(x) w_0=x s_{\b_1} \cdots s_{\b_m}$ for some positive roots $\b_1, \ldots, \b_m$. Then $x w_0 \s(x) \i=s_{x(\b_m)} \cdots s_{x(\b_1)}$ and $\ell_R(x w_0 \s(x) \i)\le m$. 
Therefore, $$ \ell_R(\CO) \le \min\{d_{\G}(x, \s(x) w_0); x \in W_0\} .$$

On the other hand, let $x \in W_0$ with $x \le \s(x) w_0$. Then we have a path in the quantum Bruhat graph from $x$ to $\s(x) w_0$ consisting of upward edges. In particular, $$d_{\G}(x, \s(x) w_0) \le \ell(x w_0)-\ell(x)=\ell(w_0)-2\ell(x).$$ Hence 
\begin{align*} \min\{d_{\G}(x, \s(x) w_0); x \in W_0\} & \le \min\{ \ell(w_0)-2\ell(x);x\le \s(x)w_0  \}\\ &=\ell(w_0)-2\max\{\ell(x);x\le \s(x)w_0\}.
\end{align*} 

Hence the second part of Theorem \ref{explicit} also follows from the inequality (a). 

\subsection{Reduction procedure: part I}\label{sec:pf1} We may write $W_0$ as
$$W_0=W_{0,1}\times W_{0,2}\times \cdots \times W_{0,l},$$
where for any $i ,$ $\s$ acts transitively on the set of irreducible components of $W_{0,i}$. Let $\CO_i$ be the $\s-$conjugacy class of the longest element $w_{0, i}$ of $W_{0,i}$. Then we have $\CO=\{ (w_1,\ldots,w_l);w_i\in \CO_i   \}$. Thus
$$\ell_R(\CO)=\sum_{i=1}^{l}\ell_R(\CO_i).$$

By definition, 
$$ \ell(w_0)-2\max\{\ell(x);x\le\s(x)w_0\}=\sum_{i=1}^l \left( \ell(w_{0, i})-2\max\{\ell(x_i);x_i\le \s(x_i)w_{0, i}\} \right).$$

Thus to prove \S\ref{sec:ineq} (a), it suffices to consider the case where $\s$ acts transitively on the set of irreducible components of $W_0$. 

\subsection{Reduction procedure: part II}\label{sec:pf2} In this subsection, we assume that $\s$ acts transitively on the set of irreducible components of $W_0$. We have $W_0=W'_0 \times \cdots \times W'_0$ with $l$ irreducible components. Moreover, we may fix the identification among the irreducible components of $W_0$ so that there exists a length-preserving group automorphism $\t$ on $W'_0$ with $$\s(w_1, w_2, \ldots, w_l)=(\t(w_l), w_1, w_2, \ldots, w_{l-1}) \text{ for all } w_1, \ldots, w_l \in W'_0.$$ 

Let $w'_0$ be the longest element of $W'_0$. Then $w_0=(w'_0, w'_0, \ldots, w'_0)$. 

\subsubsection{$l$ even case} In this case $w_0=(w'_0, 1, w'_0, 1, \ldots, w'_0, 1) \s (w'_0, 1, w'_0, 1, \ldots, w'_0, 1)$. Hence $1 \in \CO$ and $\ell_R(\CO)=0$. 

Note that if $x\le \s(x)w_0,$ then $\ell(x)\le \ell(\s(x)w_0)=\ell(w_0)-\ell(x)$ and then 
$\max\{\ell(x);x\le \s(x)w_0\} \le \frac{1}{2} \ell(w_0)$.
It is easy to see that $x=( w_0',1,w_0',1,\ldots,w_0',1)$ satisfies the condition $x\le\s(x)w_0$. Hence $\max\{\ell(x);x\le \s(x)w_0\} = \frac{1}{2} \ell(w_0) $ and $$\ell(w_0)-2\max\{\ell(x);x\le \s(x)w_0\} =0=\ell_R(\CO).$$

\subsubsection{$l$ odd case} 
Let $\CO'$ be the $\t$-conjugacy class of $w'_0$ in $W'_0$. We show that 

(a) $\ell_R(\CO)=\ell_R(\CO')$. 

Let $f: W_0 \to W'_0$ be the map sending $(w_1, \ldots, w_l)$ to $w_l w_{l-1} \cdots w_1$. Then $f$ sends a $\s$-conjugacy class in $W_0$ to a $\t$-conjugacy class in $W'_0$. By definition, $\ell_R(w_1, \ldots, w_l)=\ell_R(w_1)+\cdots+\ell_R(w_l) \ge \ell_R(w_l \cdots w_1)$. If $(w_1, \ldots, w_l) \in \CO$, then $w_l \cdots w_1 \in \CO'$. Thus $\ell_R(\CO)\ge \ell_R(\CO')$. On the other hand, for any $w\in\CO',$ it is easy to see that $(w,1,\ldots,1)\in\CO$. Then $\ell_R(\CO)\le \ell_R(\CO')$.

(a) is proved. 

Let $w'$ be a maximal length element in $\{ x'\in W_0';x'\le \tau(x')w_0'  \}$. Take $$x=(w',w'w_0',w',\ldots,w'w_0',w').$$ It is easy to see that $x\le \s(x)w_0$. Note that $\ell(w_0)-2\ell(x)=l\ell(w_0')-(l+1)\ell(w')-(l-1)\ell(w'w_0')=\ell(w_0')-2\ell(w')$.
Then $$\ell(w_0)-2\max\{\ell(x);x\le\s(x)w_0\}\le \ell(w_0')-2\max\{\ell(x');x'\le\tau(x')w_0'\}.$$

Thus to prove \S\ref{sec:ineq} (a), it suffices to consider the case where $W_0$ is irreducible.

\subsection{Irreducible cases} In the rest of this section, we calculate the number $\ell(w_0)-2\max\{\ell(x);x\le\s(x)w_0\}$ for each irreducible Weyl group. Combined \S\ref{sec:refl} with \S\ref{sec:pf1} and \S\ref{sec:pf2}, we prove \S\ref{sec:ineq} (a) and hence finish the proof of Theorem \ref{explicit}. 

If $\s=\Ad(w_0)$, then \begin{align*} \max\{\ell(x);x\le \s(x)w_0\} &=\max\{\ell(x);x\le w_0x \}=\max\{\ell(y\i);y \i\le w_0 y\i \} \\ &=\max\{\ell(y);y\le y w_0\}.\end{align*} Then $\ell(w_0)-2\max\{\ell(x);x\le \s(x)w_0\}=\ell(w_0)-2\max\{\ell(x);x\le xw_0\}$. The $\s=\Ad(w_0)$ case may be reduced to the $\s=\id$ case.

The $\s=\id$ case will be discussed in \S\ref{elt-y}. The ${}^2 D_{2k}$, ${}^3 D_4$, ${}^2 F_4$ and ${}^2I_{2k}$ cases will be discussed in \S\ref{sec:2D}--\S\ref{sec:2I}.

\subsection{Explicit construction}\label{elt-y} In this subsection we discuss the general strategy (in the $\s=\id$ case) to construct a desired element $x$ with $$x \le x w_0 \text{ and } \ell(x)=\frac{1}{2}(\ell(w_0)-\ell_R(w_0)).$$ This would imply that $\ell(w_0)-2\max\{\ell(x);x\le xw_0\}\le \ell_R(w_0)$.  
The ${}^2 D_{2k}$ case is handled by the same strategy (see \S\ref{sec:2D}). 


We use the following induction procedure. The base case is type $A_1$, where $\frac{1}{2}(\ell(w_0)-\ell_R(w_0))=0$ and we may take $x=1$. 

Let $J$ be a subset of the simple reflections and $W'_0 \subset W_0$ be the proper standard parabolic subgroup generated by $J$. Let $w'_0$ be the longest element of $W'_0$. We have $w_0=w'_0 z$ for some $z \in W_0$. By inductive hypothesis, there exists $x' \in W'_0$ with $x' \le x' w'_0$ and $\ell(x')=\frac{1}{2}(\ell(w'_0)-\ell_R(w'_0))$. 

Suppose that there exists $y \in W_0$ such that 
\begin{enumerate}
\item $y$ is the minimal length element in the coset $W'_0 y$; 

\item $z (w_0 \i y w_0)$ is the minimal length element in the coset $W'_0 z (w_0 \i y w_0)$; 

\item $y \le z (w_0 \i y w_0)$; 

\item $\ell(y)=\frac{1}{2}(\ell(w_0)-\ell_R(w_0))-\frac{1}{2}(\ell(w'_0)-\ell_R(w'_0))$. 
\end{enumerate}

Then we have $x' y w_0=x' w_0 (w_0 \i y w_0)=(x' w'_0) \bigl(z (w_0 \i y w_0)\bigr)$. By assumption on $y$, we have $x' y\le x' y w_0$ and $\ell(x' y)=\frac{1}{2}(\ell(w_0)-\ell_R(w_0))$. Thus $x' y$ is a desired element for $W_0$. 

\smallskip

Now we do the case-by-case analysis. We use the labelling of \cite{Bour}. To simplify notation, in the exceptional types, we may simple write $s_{ij\cdots}$ for $s_is_j\cdots$. Following \cite[\S 7]{He07}, for $1\le a,b\le n$, define

$$s_{[b ,a]}=\begin{cases} 
s_{b}s_{b-1}\ldots s_{a}, & \text{if } a\le b;\\
1,  &    \text{otherwise}.
  \end{cases}$$
Define $\sharp W_0=\frac{1}{2}(\ell(w_0)-\ell_R(w_0))$ and $\sharp W_0'=\frac{1}{2}(\ell(w_0')-\ell_R(w_0'))$.

The following table gives in each case an explicit subset $J$ and an explicit element $y$ we need in the induction procedure. The explicit expression for $z=(w'_0)^{-1} w_0$ is from \cite[\S 1.5]{He09}\footnote{There was a typo for type $E_6$ in \cite{He09}, which we corrected here.}. Also type $G_2$ is the same as type $I_m$ for $m=6$ and we do not put type $G_2$ case separately in the table. 

\begin{center}
\begin{tabular}{ |c|c|c|c|c|c|c|c|c| } 
 \hline
 Type & $J$ & $\sharp W_0$ & $\sharp W_0-\sharp W_0'$ & $z$ & $y$\\ 
 \hline
 $A_n(n\ge 2)$ & $ \{ 2,\ldots,n\} $   &  $ \lfloor \frac{n^2}{4}\rfloor$ &   $\lfloor \frac{ n}{2} \rfloor  $ & $s_{[n,1]}^{-1}$   & $s_{[  \lfloor \frac{n}{2}\rfloor ,1]  }^{-1}$  \\ 
 \hline
 $B_n/C_n(n\ge 2)$ & $ \{ 2,\ldots,n\} $   &  $  \frac{n^2-n}{2} $ &   $n-1  $ & $s_{[n,1]}^{-1}s_{[n-1,1]}$   & $s_{[ n-1,1 ]  }^{-1}$  \\ 
  \hline
 $D_{n} (2\nmid n,n\ge 5)$ & $ \{ 2,\ldots,n\} $   &  $ \lceil \frac{n(n-2)}{2}\rceil$ &   $n-1$ & $ s_{[n,1 ]}^{-1}s_{[n-2,1]}   $   & $s_{[ n-1,1]}^{-1}$  \\ 
        \hline
 $D_{n} (2\mid n, n\ge 4)$ & $ \{ 2,\ldots,n\} $   &  $ \lceil \frac{n(n-2)}{2}\rceil$ &   $n-2$ & $ s_{[n,1  ]}^{-1}s_{[ n-2,1]} $   & $s_{[ n-2,1]}^{-1}$  \\ 
  \hline
 $E_6$& $ \{ 1,\ldots,5\} $   &  $16$ &   $ 8  $ & $z_{E_6}$   & $y_{E_6}$  \\ 
  \hline
 $E_7$& $ \{ 1,\ldots,6\} $   &  $ 28$ &   $12 $ & $z_{E_7}$   & $ y_{E_7}$  \\ 
  \hline
 $E_8$ & $ \{ 1,\ldots,7\} $   &  $56$ &   $28  $ & $z_{E_8}$   & $y_{E_8}$  \\ 
  \hline
 $F_4$ & $ \{ 	1,2,3\} $   &  $10$ &   $7$ & $z_{F_4}$   & $y_{F_4}$  \\ 
  \hline
 $H_3$ & $ \{ 2,3  \} $   &  $  6  $ &   $  5  $ & $z_{H_3}$   & $y_{H_3}$  \\ 
  \hline
 $H_4$ & $ \{ 1,2,3  \} $   &  $  28  $ &   $  22  $ & $z_{H_4}$   & $y_{H_4}$  \\
  \hline 
 $I_m(m\ge 5)$ & $ \{ 2  \} $   &  $  \lceil\frac{m}{2}\rceil -1 $ &   $  \lceil\frac{m}{2}\rceil -1  $ & $   z_{I_m}$ & $y_{I_m}$  \\
  \hline 
\end{tabular}
\end{center}

Here
\begin{gather*}
z_{E_6}=s_{[6,1]}s_{43542}s_{[6,3]}s_1, y_{E_6}=s_{[6, 1]} s_{4 5};\\
z_{E_7}=s_{[7, 1]} s_{43542} s_{[6, 3]} s_1 s_{[7, 2]} s_{[7, 4]} \i, y_{E_7}=s_{[7, 1]} s_{43546};\\
z_{E_8}=(s_{[8,1]} s_{43542} s_{[6,3]} s_1 s_{[7,2]} s_{[7,4]} \i)^2 s_8, y_{E_8}=s_{[8,1]} s_{43542} s_{[6,3]} s_{[7,4]} s_2 s_{[8, 3]};\\
z_{F_4}=s_{[4, 1]} s_{3 2 3 4 3 2 3} s_{[4, 1]} \i, y_{F_4}=s_{[4, 1]} s_{324}; \\
z_{H_3}=s_{1 2 1 2 3 2 1 2 1 3 2 1}, y_{H_3}=s_{12312}; \\
z_{H_4}=(s_{[4,1]} s_2 s_1 s_{[3,1]} s_{23})^4 s_4, y_{H_4}=s_{[4,1]}s_{ 2 3 1 2 1 4 2 3 2 1 2 4 3 1 2 1 2 3};\\
z_{I_m}=s_{121\ldots} \text{ with } \ell(z_{I_m})=m-1, y_{I_m}=s_{121\ldots} \text{ with } \ell(y_{I_m})=\lceil \frac{m}{2}   \rceil-1.
\end{gather*}

\subsection{Type ${}^2 D_{2k}$}\label{sec:2D} By \S\ref{sec:refl}, $\ell_R(\CO)=2k-2$. Hence $$\frac{1}{2}(\ell(w_0)-\ell_R(\CO))=2k(k-1)+1.$$

If $k=2$. then $x=s_{ 1  3  2  1  3}$ is a desired element.

Now suppose $k\ge 3$. Let $W'_0$ be the standard parabolic subgroup generated by $s_2, \ldots, s_{2k}$. Then $W_0'$ is of type ${}^2D_{2k-1}$. Then $w_0=w'_0 z$ with $$z=(s_1 s_2 \cdots s_{2k-2}) s_{2k-1} s_{2k} (s_{2k-2} \cdots s_1).$$ We have $\s \mid_{W'_0}=\Ad(w'_0)$. By \S\ref{elt-y}, there exists $x'\in W_0'$ with $x'\le \s(x')w_0 $ and $\ell(x')=2(k-1)^2$. 

Take $y=s_1 s_2 \cdots s_{2k-2}s_{2k-1}$. Note that $\s(x'y)w_0=\s(x')w_0'(zw_0\s(y)w_0)=\s(x')w_0's_1 s_2 \cdots s_{2k-1}$. Hence $x'y\le \s(x'y)w_0 $ and $$\ell(x'y)=\ell(x')+\ell(y)=2(k-1)^2+2k-1=2k(k-1)+1.$$ Therefore $x'y$ is a desired element for type ${}^2 D_{2k}$.  

\subsection{Type ${}^3 D_4$}
Recall that $\ell(w_0)=12$. By \S\ref{sec:refl}, $\ell_R(\CO)= 2$. Hence $$\frac{1}{2}(\ell(w_0)-\ell_R(\CO))=5.$$ It is easy to see that $x=s_{4 3 1 2 1}$ is a desired element. 

\subsection{Type ${}^2 F_4$} By \S\ref{sec:refl}, $\ell_R(\CO)= 0$. Hence $$\frac{1}{2}(\ell(w_0)-\ell_R(\CO))=12.$$
It can be checked directly that $x=s_{2 1 3 2 4 3 2 1 3 2 4 3}$ is a desired element.

\subsection{Type ${}^2 I_{2k}$}\label{sec:2I} By \S\ref{sec:refl}, $\ell_R(\CO)= 0$. Hence $$\frac{1}{2}(\ell(w_0)-\ell_R(\CO))=k.$$
It can be checked directly that $x=s_{[1,k]}$ is a desired element.

\section{The main result}

Recall that $\mu^\diamond \in X_*(T)^{\G_0} \otimes \BQ\cong X_*(T)_{\G_0} \otimes \BQ$ is the average of the $\s$-orbit on $\mu_+$. Now we state the main result of this paper. 

\begin{theorem}\label{main} Suppose that $\underline \mu$ is superregular. Let $\CO$ be the $\s$-conjugacy class of $w_0$. Let $[b]\in B(\bG)$. If $\k(b)=\k(\mu)$ and $\mu^\diamond \ge \nu_b+2\rho^{\vee}$,  then
$$\dim X(\mu,b)=d_{\Adm(\mu)}(b)=\< \underline \mu  -\nu_b,\rho \>-\frac{1}{2}\text{def}_{\bG}(b)+  \frac{1}{2}\ell(w_0)-  \frac{1}{2}\ell_R(\CO).$$
\end{theorem}
\begin{remark}
(1) The assumption in Theorem \ref{main} implies that $[b] \in B(\bG, \mu)$. Thus by Theorem \ref{KR-conj}, $X(\mu, b) \neq \emptyset$. 

(2) We would like to point out the following special cases of Theorem \ref{main}. If $\s=\id$ (e.g., $\bG$ is split over $F$) or $\s=\Ad(w_0)$, then
$$\dim X(\mu,b)=\< \underline \mu  -\nu_b,\rho \>-\frac{1}{2}\text{def}_{\bG}(b)+  \frac{1}{2}\ell(w_0)-  \frac{1}{2}\ell_R(w_0). $$
\end{remark}

\begin{proof}
By Corollary \ref{cor:4.2}, $\dim X(\mu, b) \le d_{\Adm(\mu)}(b)$. By Proposition \ref{dim-formula} and Theorem \ref{explicit}, we have $d_{\Adm(\mu)}(b)=\< \underline \mu  -\nu_b,\rho \>-\frac{1}{2}\text{def}_{\bG}(b)+  \frac{1}{2}\ell(w_0)-  \frac{1}{2}\ell_R(\CO)$. 

By Theorem \ref{explicit}, there exists $x \in W_0$ such that $x \le \s(x) w_0$ and $\ell(w_0)-2 \ell(x)=\ell_R(\CO)$. Let $w=xt^{\underline\mu}w_0\s(x)\i$. Since $x \le \s(x) w_0$, we have $\text{wt}(x, \s(x) w_0)=0$ and thus by Proposition \ref{adm-set}, $w \in \Adm(\mu)$. In particular, $\dim X_w(b) \le \dim X(\mu, b)$. 

By \cite[Theorem 1.1 \& \S 6.4]{He20}, we have $\dim X_{w}(b)=d_{ w}(b)$. By definition, 
\begin{align*}
d_w(b) &=  \frac{1}{2}(\ell(x)-\ell(\s(x)w_0)+\ell(w_0) +\< \underline\mu,2\rho\> )  -\frac{1}{2}\text{def}_{\bG}(b) - \<\nu_b,\rho\>\\
&= \ell(x)-\frac{1}{2}\text{def}_{\bG}(b) +  \<\underline\mu-\nu_b,\rho\>\\
&= \frac{1}{2}\ell(w_0)-\frac{1}{2}\ell_R(\CO)-\frac{1}{2}\text{def}_{\bG}(b) +  \<\underline\mu-\nu_b,\rho\> \\
&= d_{\Adm(\mu)}(b).
\end{align*}

Therefore $\dim X_w(b)=d_{\Adm(\mu)}(b)$ and hence $\dim X(\mu, b)=d_{\Adm(\mu)}(b)$. 
\end{proof}

\subsection{Possible connection to Springer fibers} Finally, we would like to discuss an intriguing connection to Springer fibers. This was pointed out to us by G. Lusztig. 

We assume that the group $\bG$ is split over $F$ and the characteristic of $F$ is not a bad prime for $\bG$. Lusztig in \cite[Theorem 0.4]{Lu11} introduced the map from the set of conjugacy classes of $W_0$ to the set of unipotent conjugacy classes of $\bG(\bar F)$ Let $u$ be a unipotent element that lies in the image of $\CO$. Let $\CB$ be the flag variety of $\bG$ and $\CB_u$ be the associated Springer fiber of $u$. This is a variety over $\bar F$. Then one may deduce from \cite[Theorem 0.7(b)]{Lu11} that $\ell(w_0)-\ell_R(\CO)=2 \dim_{\bar F}(\CB_u)$. 

Combining with Theorem \ref{main}, we have that 
$$\dim X(\mu, b)=\< \underline \mu  -\nu_b,\rho \>-\frac{1}{2}\text{def}_{\bG}(b)+\dim_{\bar F}(\CB_u).$$

It is worth pointing out that $\< \underline \mu  -\nu_b,\rho \>-\frac{1}{2}\text{def}_{\bG}(b)$ is the dimension of the affine Deligne-Lusztig variety $X_{\mu}(b)$ in the affine Grassmannian (see \cite{GHKR1} and \cite{Vi06}). 

Thus we have $$\dim X(\mu, b)-\dim X_{\mu}(b)=\dim_{\bar F}(\CB_u).$$
Here the left hand side is the difference of the dimension of affine Deligne-Lusztig varieties (over the residue of $\breve F$) in the affine flag variety and the affine Grassmannian and the right hand side is the the dimension of a Springer fiber (over $\bar F$). It is very interesting to investigate the possible relation between the affine Deligne-Lusztig varieties and the Springer fibers.

\end{document}